\documentclass[12pt,reqno]{article}
\usepackage[margin=.9in]{geometry}
\usepackage{graphicx}

\usepackage{authblk}

\usepackage[small,bf,width=\textwidth]{caption}
\usepackage{nicefrac}
\usepackage{hyperref}
\hypersetup
{
    colorlinks,	%
    citecolor=blue,%
    filecolor=black,%
    linkcolor=blue,%
    urlcolor=blue
}
\usepackage{amssymb,amsmath,amsthm,amscd}
\usepackage{eulervm, xfrac}
\usepackage[sans]{dsfont}
\usepackage{mathdots}
\usepackage{mathrsfs}
\usepackage{rotating, setspace}
\DeclareMathAlphabet{\mathbfsf}{\encodingdefault}{\sfdefault}{bx}{n}
\usepackage[all]{xy}
\usepackage{enumerate}
\DeclareFontShape{OT1}{cmr}{bx}{sc}{<-> cmbcsc10}{}

\usepackage{sectsty}
\allsectionsfont{\sffamily\bfseries}

\usepackage{comment}
\usepackage{authblk}

\usepackage{soul}
\usepackage{todonotes}

\usepackage{multirow}%

\usepackage{mathtools}
\usepackage[title]{appendix}%
\usepackage{xcolor}%
\usepackage[capitalise]{cleveref}
\usepackage{quiver}%
\usepackage{textcomp}%
\usepackage{manyfoot}%
\usepackage{booktabs}%
\usepackage{algorithm}%
\usepackage{enumerate}   
\usepackage{algorithmicx}%
\usepackage{algpseudocode}%
\usepackage{listings}%

\usepackage{multirow}
\usepackage{bbm}

\usepackage{thmtools}
\usepackage{thm-restate}

\theoremstyle{definition} 
\newtheorem{definition}{Definition}[section]

\newtheorem{observation}[definition]{Observation}
\newtheorem{example}[definition]{Example}

\theoremstyle{plain}
\newtheorem{theorem}[definition]{Theorem}
\newtheorem{lemma}[definition]{Lemma}
\newtheorem{corollary}[definition]{Corollary}

\newcommand{\N}{\mathbb{N}}
\newcommand{\R}{\mathbb R}
\newcommand{\Z}{\mathbb Z}

\newcommand{\barc}{B}
\newcommand{\Barc}{\textbf{Barc}}
\newcommand{\JS}{\textbf{JS}}
\newcommand{\Betti}{\textbf{Betti}}
\newcommand{\define}[1]{\textbf{#1}}

\begin{document}

\title{\textsf{Counting Barcodes with the same Betti Curve}}

\author{\textsf{Henry Ashley, H\r{a}vard Bakke Bjerkevik, Justin Curry,} \\ \textsf{Riley Decker, Robert Green}\footnote{Department of Mathematics and Statistics, SUNY Albany}}

\date{\textsf{\today}}


















\maketitle

\abstract{This paper considers an important inverse problem in topological data analysis (TDA): How many different barcodes produce the same Betti curve? 
    Equivalently, given a function \[\beta\colon [n]=\{1<\cdots< n\} \to \mathbb{Z}_{\geq 0},\] how many different ways can we write $\beta$ as a sum of indicator functions supported on intervals in $[n]$?
    Our answer to this question is to connect persistent homology with the study of the Kostant partition function and the enumerative combinatorics for so-called ``magic'' juggling sequences studied by Ronald Graham and others. 
    Specifically, we prove an equivalence between our inverse problem and corresponding statements in these other two settings.
    From an applications and statistics point of view, our work provides a quantification of how lossy the TDA pipeline is when moving from persistent homology to persistent Betti numbers.}



\section{Introduction}
\subsection{Context and Prior Work}\label{sec:context}

Topological Data Analysis (TDA) provides an impressive toolkit for the extraction and summary of shape in data. 
The input to the TDA pipeline can be one of several types of data: a finite subset of $\mathbb{R}^d$, a semi-algebraic set, a scalar or vector field on a manifold, and so on.
The output of the TDA pipeline can also be one of several types of summary objects: a merge tree, a Reeb graph, a persistence diagram/barcode, or a piecewise constant $\mathbb{Z}$-valued function, such as a Betti curve or Euler curve.
By fixing the class of input data and the type of output object, one naturally encounters an \emph{inverse problem}: How far is the TDA pipeline from being injective for that class of data set?
Since topology necessarily ignores certain aspects of a data set, it is to be expected that this mapping is not injective and typically will have high-dimensional, geometrically complex fibers\footnote{As a reminder, the \define{fiber} of a mapping $f\colon Y\to X$ over a point $x\in X$ is its pre-image $f^{-1}(x)$.}.
This is evidenced by a robust line of research into inverse problems in TDA, of which \cite{gameiro2016continuation,gasparovic2018complete,curry2018fiber,kanari2020trees,catanzaro2020moduli,oudot2020inverse,desha2021inverse,curry2022many,leygonie2022fiber,leygonie2022differential,garin2022trees,mallery2022lattice,solomon2023geometry,beers2023fiber,beers2023topology,curry2024trees,smith2024generic} is a likely incomplete sampling.

This paper extends existing research~\cite{curry2018fiber,kanari2020trees,garin2022trees,mallery2022lattice,curry2024trees} into inverse problems between TDA objects themselves, which tends to have a more combinatorial flavor. 
Those papers were mostly concerned with understanding how $\pi_0$, the \emph{set} of connected components, refines $H_0$, the \emph{vector space} freely generated by the set of connected components.
In the persistence setting, this leads to the now well-understood mapping from merge trees to barcodes.
The problem considered here is, in some aspects, similar as we consider how lossy the process of replacing a vector space with its dimension is.

\subsection{Problem Statement}
\label{sec:problem_statement}

In this paper we formulate and shed light on a hitherto unconsidered inverse problem between two well-studied TDA objects---barcodes and Betti curves---in the finitely indexed setting.
To explain, recall the following: Given any functor 
\[
F\colon \mathbf{[n]}=\{1 < \cdots < n\} \to \mathbf{vect}
\]
from the totally ordered set on $n$ elements to the category of finite-dimensional vector spaces, one has an associated dimension (or Hilbert) function 
\[
\dim(F)\colon [n]\to \mathbb{Z}_{\geq 0},
\]
which assigns to each $i\in[n]$ the dimension of the vector space $F(i)$.
Our inverse problem is the following:
\begin{quote}
\normalsize
    {\it \underline{Given:} a dimension function $h\colon [n]\to \mathbb{Z}_{\geq 0}$.}
\vspace{.1in}

    \noindent {\it \underline{Enumerate:} all functors $F\colon\mathbf{[n]} \to \mathbf{vect}$, up to iso, with $\dim F = h$.}
\end{quote}
Equivalently, a functor $F\colon\mathbf{[n]} \to \mathbf{vect}$ is a type $A_n$ quiver representation, whose isomorphism class is determined by its decomposition into interval representations.
The dimension function and the multiset of intervals resulting from the decomposition are how we get our notions of Betti curves and barcodes.


\begin{definition}
A \define{Betti curve of length $n$} is a function $\beta\colon [n] \to \Z_{\geq 0}$, or equivalently, a finite sequence $(\beta_1, \beta_2, \ldots, \beta_{n})$ where $\beta_i=\beta(i)$.
We define $\Betti(n)$ as the set of Betti curves of length $n$. Since $\beta(i)\in \Z_{\geq 0}$ can be arbitrary, $\Betti(n)$ is necessarily an infinite set.
\end{definition}


For $i\leq j$, we use the notation $[i,j)$ for the set $\{i,i+1,\dots,j-1\}$, and we refer to nonempty sets of this form as \define{intervals}.
For instance, $[n] = [1,n+1)$ by this convention.


\begin{definition}
A \define{barcode} is a multiset of intervals.
In other words, a barcode is a set 
\[
\barc= \{(I_j,m_j) \mid I_j \text{ is an interval in }[n] \text{ and } m_j\in \Z_{\geq 0} \text{ indicates multiplicity}\}
\]
We define $\Barc(n)$ as the set of barcodes whose intervals are subsets of $[n]$.
\end{definition}


Observe that we can associate a Betti curve $\beta$ to a barcode $\barc$ by letting $\beta(i)$ be the number of intervals in $\barc$ that contain $i$.
More precisely, let $\barc= \{(I_j,m_j)\}$. 
We define the function $\varphi\colon \Barc(n) \to \Betti(n)$ that associates a Betti curve to a barcode in the following way:
\begin{align*}
\varphi(\barc) &=  \beta,\\
\beta(i) &=\sum_{j \mid I_j\cap \{i\} \neq \varnothing} m_j.
\end{align*}
Note that $\varphi$ is surjective: for $\beta=(\beta_1, \beta_2, \ldots, \beta_{n})$, we can construct $\barc\in \varphi^{-1}(\beta)$ by letting $\barc$ have $\beta_i$ copies of $[i,i+1)$ for each $i$.
We define 
\[
\Barc(\beta) \coloneqq \varphi^{-1}(\beta). 
\]
In other words, $\Barc(\beta)$ is the set of barcodes whose Betti curve is $\beta$; it characterizes the fiber of the dimension function for all type $A_n$ quiver representations.
Therefore, our inverse problem is equivalent to the following: 

\begin{quote}
\normalsize
    {\it \underline{Given:} a Betti curve $\beta\colon[n]\to \mathbb{Z}_{\geq 0}$.}
\vspace{.1in}

    \noindent {\it \underline{Enumerate:} all barcodes $B$ in $\Barc(\beta)$.}
\end{quote}

For a concrete example of our enumeration problem, if $n=3$ and one has the Betti curve $\beta=(2,3,2)$, then there are 13 different barcodes with this Betti curve.

\begin{figure}[h]
    \centering
    \captionsetup{width=.85\textwidth}
    \includegraphics[width=.85\textwidth]{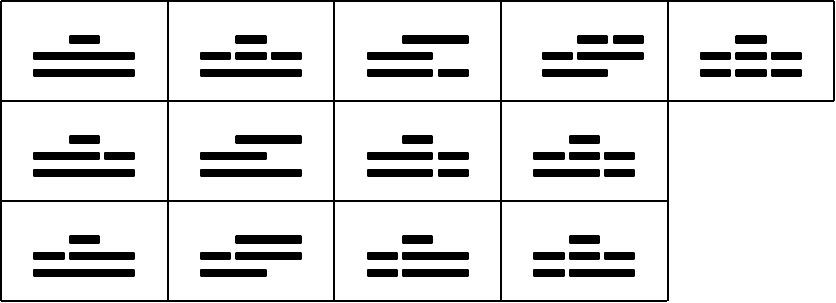}
    \caption{A table of all the barcodes with the Betti curve $\beta =(2,3,2)$. Here each bar is its own interval and bars at the same height are there only for visual convenience.}
    \label{fig:232_barcode_table}
\end{figure}

\subsection{Main Contributions}\label{sec:contributions}

We provide a solution to the barcode counting problem in several ways: 
First, we reinterpret the problem through the lens of Lie algebras, witnessing that the answer to our question in the type $A_n$ Lie algebra setting is exactly the Kostant partition function $K$. 

\begin{restatable}{thmx}{barcCountingKostant}\label{thm:barc_counting_is_kostant}
    If $\mu= \sum\limits_{i=1}^{n} \beta_i \alpha_{i}$ is a weight of the Lie algebra of type $A_n$, then 
    \[|\mathbf{Barc}(\beta)| = K(\mu).\]
\end{restatable}

In the above, we use $\alpha_i$ to denote the simple root $\mathbf{e}_i - \mathbf{e}_{i+1}$ in a root system of type $A_n$.
The Kostant partition function is a famous object in representation theory and algebraic combinatorics; it counts the number of ways a weight of a Lie algebra can be partitioned into a non-negative integral linear combination of positive roots.
The appearance of the Kostant partition function becomes completely intuitive when one recalls Gabriel's Theorem, which states that the indecomposables of a Dynkin quiver representation are in one-to-one correspondence with the positive roots of the root system for the corresponding Dynkin diagram. Since a persistence module with field coefficients is an $A_n$ quiver representation, this relationship between barcodes and combinations of positive roots is unsurprising.

In Section \ref{sec_recursive}, we focus on computational aspects of this inverse problem, and provide an explicit recursive formula for enumerating the number of elements of $\Barc(\beta)$. This formula was first observed by Schmidt and Bincer \cite{SchmidtBincer} in the context of the Kostant partition function; we provide an alternative proof of the result, along with a pictorial interpretation.  

In \Cref{sec:bijection}, we establish a novel bijection between barcodes and so-called ``magic multiplex'' juggling sequences~\cite{Multiplex}.
A juggling sequence is a sequence of juggling states, and a juggling state is a vector $\textbf{s}=\langle s_1,s_2, \dots, s_h\rangle$, which can be interpreted as the heights of all the balls currently in the air.
If $s_3=2$, then there are two balls in the air that we will catch in three time steps and immediately throw up in the air again.
Negative numbers in a juggling state vector are allowed---these will eventually cancel out balls thrown to appropriate heights.
Let $\JS(\textbf{a},\textbf{b},n)$ denote the set of all magic juggling sequences of length $n$ with initial state \textbf{a} and terminal state \textbf{b}, and let $\JS(\leq n+1,\textbf{b},n)$ be the set of magic juggling sequences of length $n$ with terminal state $\textbf{b}$ and whose initial state is zero after the first $n+1$ entries.
For a Betti curve $\beta$, let $\delta(\beta)$ be the vector with $\delta(\beta)_i = \beta_i-\beta_{i-1}$, where we set $\beta_0=0$.
We define a map $\sigma$ sending barcodes to juggling sequences and prove the following.

\begin{restatable}{thmx}{bijectionThm}\label{thm:bijection}
The map $\sigma\colon \Barc(n) \to \JS(\leq n+1,\langle 0 \rangle, n)$ is a bijection.
Moreover, for every Betti curve $\beta=(\beta_1, \beta_2, \dots, \beta_{n})$, $\sigma$ restricts to a bijection $\sigma_\beta\colon \Barc(\beta) \to \JS(\langle \delta(\beta) \rangle,\langle 0 \rangle, n)$.    
\end{restatable}

This then connects TDA to the historically rich area of the mathematics of juggling~\cite{Buhler1994,Butler2017,Chung,Ehrenborg1996,MagicJuggling,PrimeJuggling,shannon1993scientific,polster2003mathematics}.
The basic idea of our bijection is that an interval in a barcode that is born at time $i$ and dies at time $j$ is equivalently viewed as a juggling move where a ball at time $i$ is thrown to height $j-i$.
A Betti curve is then equivalently a count of the number of balls in the air at any time $i$.

Moreover, by building on \cite{MagicJuggling}, we recover \Cref{thm:barc_counting_is_kostant} as a corollary of \Cref{thm:bijection}.
The culmination of our work is thus a modern Rosetta stone between TDA, Juggling, and the study of combinatorial aspects of Lie algebras, shown below.
\vspace{5mm}

\begin{center}
        \begin{tabular}{|c|c|c|}
        \hline
             TDA object & Juggling analog & Lie algebra analog\\
        \hline
        \hline
            Interval module & \multirow{2}{*}{Throw at time $i$ to height $j-i$}  & Positive root of $A_n$ \\
            $\mathbbm{k}[i,j)$ & & $\mathbf{e}_i - \mathbf{e}_j$\\
        \hline
        Barcode & Juggling sequence & Partition of a weight \\
        \hline
            Betti curve & Collection of juggling sequences & Weight of $A_n$ \\
        \hline
            $|\mathbf{Barc}(\beta)|$ & $|\JS(\langle\delta(\beta)\rangle, \langle 0\rangle, n)|$ & $K\Big(\sum\limits_{i=1}^n \beta_i \alpha_i\Big)$\\
        \hline
        \end{tabular}
\end{center}

\vspace{5mm}


\section{Background}
\subsection{Representation Theoretic Aspects of TDA}\label{sec_rep_theory}

    Our basic objects of study in TDA are \define{persistence modules}, i.e., functors $F\colon \mathbf{[n]} \to \mathbf{vect}$. 
    Alternatively, we can view these functors as examples of \define{representations} of the type $A_n$ Dynkin quiver
    \[
    \bullet \to \bullet \to \dots \to \bullet,
    \]
    which is obtained by replacing each vertex with a vector space and each arrow with a linear transformation.
    In other words, it has the form
    \[
    V_1 \xrightarrow{\phi_1} V_2 \xrightarrow{\phi_2} \dots \xrightarrow{\phi_{n-1}} V_n.
    \]
    In representation theoretic language, a Betti curve is simply a dimension vector, i.e. the vector associated to a representation which records the dimension of the vector spaces at each index. 
    
    When the vector spaces are finite-dimensional, such representations are of a particularly nice form, thanks to the following two theorems.
    \begin{theorem}[Azumaya-Krull-Remak-Schmidt]{\cite{azumaya1950corrections}}
    Let $Q$ be a finite quiver, and let $\mathbbm{k}$ be a field. Every finite dimensional representation $\mathbb{V}$ of $Q$ over $\mathbbm{k}$ decomposes uniquely (up to isomorphism and permutation of indecomposable summands) as a direct sum 
    \[
    \mathbb{V}^1 \oplus \mathbb{V}^2 \oplus \dots \oplus \mathbb{V}^m,
    \]
    where each $\mathbb{V}^i$ is indecomposable.
    \end{theorem}
    
    \begin{theorem}[Gabriel's Theorem for type $A_n$ quivers]{\cite{gabriel1972unzerlegbare}}
    Let $Q$ be a type $A_n$ quiver, and let $\mathbbm{k}$ be a field. 
    \begin{enumerate}
        \item Every indecomposable, finite-dimensional representation of $Q$ is isomorphic to one of the form\footnote{The directions of the arrows in this statement actually depend on their direction in $Q$. We choose all arrows of $Q$ to go forward, since representations of this quiver are persistence modules indexed by the totally ordered set $\{1< \ldots < n\}$.}
    \[
    0 \to \dots \to  0 \to \mathbbm{k} \xrightarrow{\text{id}} \mathbbm{k} \xrightarrow{\text{id}} \dots \xrightarrow{\text{id}} \mathbbm{k} \to 0 \to \dots \to 0 .
    \]
    \item Moreover, these indecomposables are in bijection with the positive roots of the $A_n$ root system, constructions that are reviewed below.
    \end{enumerate}
    \end{theorem}
    
    We call such indecomposables \define{interval representations}, and call the first index of $\mathbbm{k}$ the \define{birth} of the interval representation, and the first zero vector space after the birth the \define{death} of the interval representation. We use $\mathbbm{k}[i,j)$ to denote an interval representation with birth $i$ and death $j$.
    These theorems are what give us our definition of barcodes. Since every type $A_n$ quiver representation decomposes (up to isomorphism) as a direct sum of interval representations, we can completely characterize the isomorphism type of the representation with a multiset of birth-death pairs coming from the intervals.

    Often overlooked by the TDA community, the second part of Gabriel's Theorem gives us a direct connection to Lie algebras, which we now explain.
    
\subsection{Lie Algebras and the Kostant Partition Function}\label{sec:lie_alg}
Because we are interested in the Lie algebra of type $A_n$, we state our definitions at this level of specificity. For further background on Lie algebras and their representation theory, see \cite{humphreys2012introduction}.

\begin{definition}
    A \define{root system} of type $A_n$ is the set of vectors 
    \[
    \Phi_{A_n}=\{ \pm(\mathbf{e}_i - \mathbf{e}_j) : 1 \leq i < j \leq n+1\},
    \] 
    where each $\mathbf{e}_i$ is a standard basis vector. Elements of $\Phi_{A_n}$ are called \define{roots}, and we use $\Phi_{A_n}^+\subset \Phi_{A_n}$ to denote the set of positive roots, that is, the roots of the form $\mathbf{e}_i - \mathbf{e}_j$ with $i< j$. The \define{simple roots} of $\Phi_{A_n}$ are 
    \[
        \{\alpha_i = \mathbf{e}_i - \mathbf{e}_{i+1} : 1\leq i \leq n\}.
    \]
    We note that every element of $\Phi_{A_n}^+$ can be written uniquely as a positive sum of simple roots. 
    In particular, $\mathbf{e}_i - \mathbf{e}_j = \alpha_i + \ldots + \alpha_{j-1}$.
\end{definition}

\begin{definition}
    An element $\mu \in \R^{n+1}$ is called a \define{weight} if it lies in the positive cone of the root lattice, that is, $\mu = \sum_{i=1}^n z_i \alpha_i$ for $z_i \in \mathbb{Z}_{\geq 0}$.
\end{definition}

Although weights are typically defined more broadly, this definition is sufficient for our purposes.

\begin{definition}
    The \define{Kostant partition function} $K$ is a function that counts the number of ways a weight $\mu$ can be written as a non-negative integral linear combination of positive roots. 
\end{definition}
The Kostant partition function was originally used to establish a formula for the multiplicity of a weight in an irreducible representation of a complex semisimple Lie algebra \cite{kostant1958formula}, but appears in various combinatorial settings, such as in the study of juggling sequences \cite{MagicJuggling} and flow polytopes \cite{meszaros2015flow}.

\begin{example}
The weight $\mu = \alpha_1 + 2\alpha_2 + \alpha_3$ can be written in 5 distinct ways (with each chosen positive root enclosed in parentheses): 
\begin{align*}
    \mu & = (\alpha_1 + \alpha_2+\alpha_3) + (\alpha_2) \\
        &=  (\alpha_1+\alpha_2) +(\alpha_2) +(\alpha_3)\\
        &=  (\alpha_{1}) + (\alpha_2) + (\alpha_2+\alpha_3) \\
        &=  (\alpha_1+\alpha_2) + (\alpha_2+\alpha_3)\\
        &=  (\alpha_{1}) + (\alpha_2) + (\alpha_2) + (\alpha_3)\\
\end{align*}
Therefore, $K(\mu)=5$.
\end{example}

We can now prove \Cref{thm:barc_counting_is_kostant}.

\barcCountingKostant*
\begin{proof}
First, we note that the bijection between indecomposable $A_n$ quiver representations and positive roots in the corresponding root system, mentioned as Part 2 of Gabriel's Theorem, is given by the (Hilbert) dimension function $\dim$.
More precisely, the interval representation $\mathbbm{k}[i,j)$ is sent to the dimension vector 
\[(0, \ldots, 0, \underbrace{1, 1, \ldots, 1}_{[i,j-1]}, 0, \ldots, 0)\]
in the simple root basis, that is, the positive root $\alpha_i + \ldots + \alpha_{j-1}= \mathbf{e}_i - \mathbf{e}_{j}$.
Since we can identify each positive root $\mathbf{e}_i - \mathbf{e}_j$ with the interval $[i,j)$, the number of ways the weight \[\mu=\sum_{i=1}^{n} \beta_i \alpha_i\] can be partitioned into a sum of positive roots is exactly equal to the number of ways $\beta\colon [n] \to \mathbb{Z}_{\geq 0}$ can be written as a sum of indicator functions supported on intervals.
\end{proof}



\subsection{A Recursive Formulation}\label{sec_recursive}

Most modern approaches to evaluating $K(\mu)$ involve modeling the problem as counting the number of integer lattice points of a polytope defined as a function of $\mu$.
The most famous technique for counting integer lattice points is Barvinok's 1994 algorithm \cite{barvinok1994}.
Lattice point counting can be carried out with \textbf{Sage} \cite{sagecalculation}, or more specialized software like \textbf{barvinok} \cite{barvinok2004implementation} or \textbf{LattE} \cite{Latte2004}. 
Since the polytopes that encode the Kostant partition function are flow polytopes, more specialized computations are possible \cite{baldoni2008flow,meszaros2015flow}.
For an overview of  methods in computing the Kostant partition function, see \cite{ModernComputation}.
For the remainder of this section, we will present and explain a recursive method of evaluating the Kostant partition function, applicable to a type $A_n$ Lie algebra.
A central technique in this method is identifying Young diagrams that ``fit under'' a Betti curve.

Let $X=(x_1, x_2, \ldots, x_{n})$ and $Y=(y_1, y_2, \ldots, y_{n})$ be sequences of non-negative integers. We say $Y\sqsubset X$ if
\begin{enumerate}
    \item $Y$ is non-increasing
    \item $y_1=x_1$
    \item $y_i\leq x_i$ for all $i\neq 1$
\end{enumerate}
Let $X+Y$ and $X-Y$ denote the coordinate-wise sum and difference of finite sequences, respectively.
\begin{theorem}\cite[Equation 4.21]{SchmidtBincer}
    Let $\beta$ be a Betti curve. The following recursive formula holds. $$ |\Barc(\beta)| = \sum_{Y\sqsubset \beta} |\Barc(\beta - Y)|$$
\end{theorem}
Note that each $\beta - Y$ is a sequence that begins with a zero, so it can be regarded as a Betti curve of length $n-1$.
The notation we use is different from that used by Schmidt and Bincer in their article \cite{SchmidtBincer}.
Our equation is demonstrated by Equation 4.21 in their paper, with conditions 2 and 3 illuminated by their Equations 4.20 and 4.22, respectively.
Just as we reduce a Betti curve of size $n$ to a sum over Betti curves of size $n-1$, they reduce a partition function of a weight in a type $A_n$ Lie algebra to a sum of partition functions over weights in a type $A_{n-1}$ Lie algebra.
For convenience, we give a short, self-contained proof of the theorem.

\begin{proof}
    Recall the function $\varphi\colon \Barc(n) \to \Betti(n)$ that associates a Betti curve to a barcode, defined in the introduction.
    Define the equivalence relation $\sim$ on elements of $\varphi^{-1}(\beta)$ by $\barc \sim \barc'$ if $\barc$ and $\barc'$ have the same multiset $\mathcal Y$ of intervals that begin at $1$.
    We write $\Theta_{\mathcal Y}$ for the corresponding equivalence class.
    For $Y\sqsubset \beta$, there is a unique multiset of intervals $\overline{Y}$ such that all the intervals start at 1 and $\varphi(\overline{Y})=Y$.
    It follows that there is a bijection between equivalence classes of $\sim$ and the set of $Y$ with $Y\sqsubset \beta$.
    Thus, to prove the lemma, it suffices to construct a bijection $\gamma\colon \Theta_{\overline{Y}}\to \varphi^{-1}(\beta - Y)$ for each $Y\sqsubset \beta$.
    Let $\gamma(X) = X - \overline{Y}$.
    The inverse of $\gamma$ is clearly given by $\gamma^{-1}(Z) = Z\sqcup\overline{Y}$, assuming both are well-defined.
    The following computation shows that $\gamma$ is well-defined:
    \[
    \varphi(X - \overline{Y})=\varphi(X) - \varphi(\overline{Y}) = \beta - Y.
    \]
    For $Z$ in $\varphi^{-1}(\beta - Y)$, we have
    \[
    \varphi(Z\sqcup\overline{Y})=\varphi(Z) + \varphi(\overline{Y}) = \beta - Y + Y = \beta,
    \]
    so $\gamma^{-1}(Z)\in \varphi^{-1}(\beta)$.
    Since $Z$ has no intervals starting at $1$, we get $\gamma^{-1}(Z) \in \Theta_{\overline{Y}}$.
\end{proof}

Intuitively, a sequence $Y\sqsubset\beta$ arises from a multiset of intervals beginning at 1 from a barcode which has Betti curve $\beta$.
Since there are exactly $\varphi^{-1}(\beta-Y)$ many barcodes in $\varphi^{-1}(\beta)$ which have this property, the formula counts every element of the this set exactly once.
We can view the $\sqsubset$ condition pictorially: $Y\sqsubset \beta$ when $Y$ is a Young diagram that ``fits inside of'' $\beta$ such that the number of cells in the first columns of $Y$ and $\beta$ is the same.

\begin{figure}[h]
    \centering
    \captionsetup{width=0.8\textwidth}
    \includegraphics[width=0.8\textwidth]{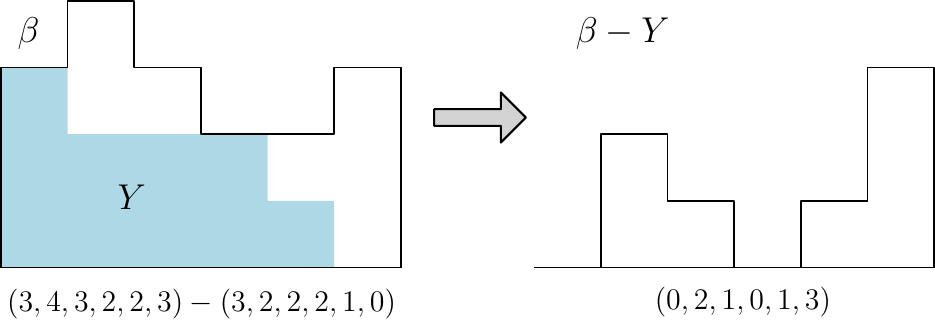}
    \caption{An example of a Betti curve $\beta$ containing a sequence $Y\sqsubset \beta$ (left), and the resulting Betti curve $\beta - Y$ (right).}
    \label{fig:placeholder}
\end{figure}

\subsection{Juggling Sequences}\label{sec_juggling}

The mathematical study of juggling dates back to at least the 1980s, when Claude Shannon wrote ``Scientific Aspects of Juggling,'' which was later published as part of this collected works \cite{shannon1993scientific}. 
Shannon's paper is filled with ideas, commenting on historical, physical, mathematical, and robotic aspects of juggling.
However, the recent focus of mathematical juggling, largely stemming from \cite{Buhler1994}, aims to enumerate different ways to theoretically juggle with various restrictions in place \cite{PrimeJuggling,Multiplex,Chung,polster2003mathematics}. 
In the study of juggling, one considers \emph{juggling states} (or \emph{landing schedules}), vectors whose entries represent (a discrete set of) times at which a number of balls will reach the juggler's hand.
A \emph{juggling sequence} is a sequence of juggling states, which represents the act of juggling over time. Consecutive states in a juggling sequence must have a so-called \emph{valid transition} between them, whose definition varies depending on the juggling variation being considered.
In the following, we adhere to the magic juggling constraints given by \cite{MagicJuggling}.

\begin{definition}
We use the notation $\langle s_1,s_2, \dots, s_h\rangle$ for the infinite-length zero-padded vector $(s_1,s_2, \dots, s_h, 0,0,\dots)$.
A \define{juggling state} is a vector $\textbf{s}=\langle s_1,s_2, \dots, s_h\rangle$, where $s_i\in \mathbb{Z}$ and such that $\sum\limits_{j=1}^h s_j=m\geq 0$. A \define{magic juggling sequence} of length $n$ is an ordered list $S=(\mathbf{s}^0, \mathbf{s}^1, \dots, \mathbf{s}^n)$ of juggling states with valid transitions between consecutive states.
A \define{valid transition} from state $\mathbf{s}^{i-1}=\langle s_1, s_2, \dots, s_h \rangle $ to state $\mathbf{s}^{i}$ occurs when $\mathbf{s}^i=\langle s_2 +t_1 , s_3+ t_2, \dots, s_h+t_{h-1}, t_h, \dots, t_{h'}\rangle$, where each $t_j$ is a non-negative integer such that $\sum_{j=1}^{h'} t_j =s_1$. 
We say $t_j$ is the number of throws at time $i$ to height $j$.
Note that the collection of $t_j$'s are always defined relative to a transition from $\mathbf{s}^{i-1}$ to $\mathbf{s}^{i}$, so the time-index $i$ is implied.
\end{definition}

\begin{figure}
    \centering
    \captionsetup{width=.87\textwidth}
    \includegraphics[width=0.65\textwidth]{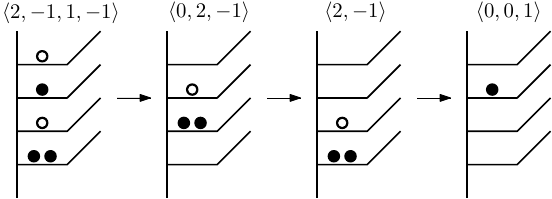}
    \caption{A magic juggling sequence represented with a ``bucket'' diagram. At each consecutive state in the juggling sequence, every ball descends by 1 bucket, except for those at the bottom bucket, which are redistributed. Magic balls appear in white and indicate a negative entry, while ordinary balls appear in black and indicate a positive entry. When a magic ball and ordinary ball occupy the same bucket, they cancel out. In this example, the two balls at time 1 are thrown to heights 1 and 2, and the former cancels with a magic ball. No balls are thrown at time 2. At time 3, the two balls at the bottom level are thrown to heights 1 and 3, and the former cancels with the one remaining magic ball.}
    \label{fig:JS_solo_example}
\end{figure}

We make the following convenient observation that completely characterizes valid magic juggling sequences.
\begin{observation}
\label{obs_valid_sequence}

Let $S=(\mathbf{s}^0, \mathbf{s}^1, \dots, \mathbf{s}^n)$ be a sequence of juggling states, and write $\mathbf{s}^i = \langle s^i_1,\dots,s^i_{h} \rangle$, and $s^i_{k} = 0$ for all $k>h$ and $k<1$.
Note that we can choose an appropriate $h$ for all states by padding with zeros.
Then $S$ is a valid magic juggling sequence if and only if
\begin{enumerate}[(i)]
\item for all $0\leq i,j\leq n$, we have $\sum_{k=1}^{h} s^i_k = \sum_{k=1}^{h} s^j_k$, and
\item for all $1\leq i\leq n$ and $1\leq k$, we have $s^{i-1}_{k+1} \leq s^{i}_{k}$.
\end{enumerate}
\end{observation}

\begin{definition}
   We let $\JS(\textbf{a},\textbf{b},n)$ denote the set of all magic juggling sequences of length $n$ with initial state \textbf{a} and terminal state \textbf{b}.
   We let $\JS(\leq h,\textbf{b},n)$ denote the set of all magic juggling sequences of length $n$ with terminal state \textbf{b}, but whose initial state is any state of the form $\langle s_1,\dots,s_{h} \rangle$; that is, the initial state is zero after the first $h$ entries. 
\end{definition}

The authors of \cite{MagicJuggling} provide the following result relating magic juggling sequences and the Kostant partition function. 
\begin{theorem}[{\cite[Theorem 3.8]{MagicJuggling}}]\label{thm:magic_juggling_KPF}
    Let $\mu$ be a weight of the Lie algebra of type $A_n$ and let $(\mu_1, \ldots, \mu_{n+1})$ be its standard basis vector representation. Then 
    \[ K(\mu) = |\JS(\langle \mu_1, \ldots, \mu_n \rangle, \langle \mu_1 + \ldots + \mu_n \rangle, n)|. \]
\end{theorem}


\section{The Bijection: Barcodes and Juggling Sequences}\label{sec:bijection}

\begin{definition}
    Let $\beta$ be a Betti curve. The \define{differential} of $\beta$, denoted $\delta(\beta)\colon [n+1] \to \Z$, is the finite sequence defined via 
    \[
    \delta(\beta) = (\beta_1, \beta_2-\beta_1, \dots, \beta_{n}-\beta_{n-1}, -\beta_{n}).
    \]
    For a barcode $\barc$ with associated Betti curve $\varphi(\barc)=\beta$, we will sometimes abuse our differential notation and write $\delta(\barc)$ instead of $\delta(\varphi(\barc))$.
\end{definition}

\begin{definition}
Given a barcode $\barc$ and $i\in \N$, define the \define{$i$-truncation} of $\barc$ as the barcode $\tau_i(\barc)$ with a bar $[a-i,b-i)$ for every bar $[a,b)$ in $\barc$ with $a\geq i+1$.
\end{definition}

Said simply, to construct $\tau_i(\barc)$, we take the bars in $\barc$, remove those whose left endpoint is at most $i$, and shift the remaining bars $i$ steps to the left.

\begin{definition}
For $0\leq i\leq n$, we write $\barc^i = \tau_i(\barc)$.
We define the \define{juggling map}
\[
\sigma\colon \Barc(n) \to \JS(\leq n+1,\langle 0 \rangle, n)
\]
that sends a barcode $B$ to its corresponding magic juggling sequence
\[
\sigma(\barc)\coloneqq(\langle \delta(\barc^0)\rangle, \dots, \langle\delta(\barc^n)\rangle).
\]
Proving that this is well-defined is part of \cref{thm:bijection}.
\end{definition}

\begin{figure}[h]
    \centering
    \includegraphics[width=0.65\textwidth]{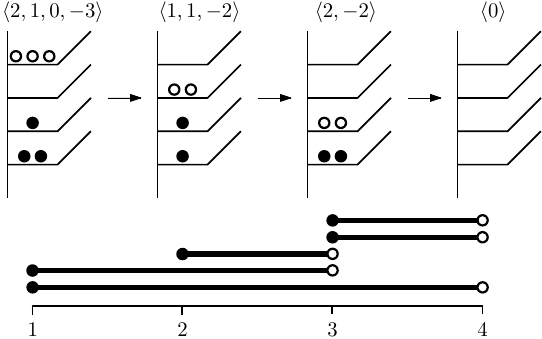}
    \caption{A barcode and its associated juggling sequence under the map $\sigma$.}
    \label{fig:JS_vs_barcode}
\end{figure}

\begin{lemma}
\label{lemma_bijection}
Let $\barc$ be a barcode with Betti curve $\beta=(\beta_1, \beta_2, \dots, \beta_{n})$, and let $b_{i,j}$ denote the multiplicity of the interval $[i,j)$ in $\barc$.
Write $\delta(\barc^i)=\langle s^i_1, s^i_2, \dots, s^i_{n+1}\rangle$, and let $s^i_k=0$ for $k<1$ and $k>n+1$. The following three statements hold:
\begin{enumerate}[(i)]
\item $\delta(\barc^0) = \delta(\beta)$
\item $\barc^n$ is the empty barcode, so $\delta(\barc^n) =  0$
\item $s^{i}_{j} - s^{i-1}_{j+1} = b_{i,i+j}$ for $i\in \{1, \ldots, n\}$ and all $j\geq 0$. 
\end{enumerate}
\end{lemma}

\begin{proof}
Observations (i) and (ii) are clear. To see (iii), notice that an interval in $\barc$ contributes
\begin{itemize}
\item positively to both $s^{i}_{j}$ and $s^{i-1}_{j+1}$ if it starts at $i+j$,
\item negatively to $s^{i-1}_{j+1}$ if it is of the form $[a,i+j)$ with $a\geq i$, and 
\item negatively to $s^{i}_{j}$ if it is of the form $[a,i+j)$ with $a\geq i+1$.
\end{itemize}
Otherwise, it does not contribute to either.
Thus, $s^{i}_{j} - s^{i-1}_{j+1}$ is exactly the number of bars in $\barc$ of the form $[i,i+j)$.
\end{proof}

\begin{figure}[h]
    \centering
    \includegraphics[width=0.5\textwidth]{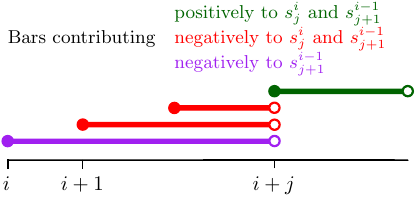}
    \caption{An illustration of how various bars contribute to $s^{i}_{j}$ and $s^{i-1}_{j+1}$.}
    \label{fig:barcode_state_contribution}
\end{figure}

\bijectionThm*
The theorem can be summarized by the commutative diagram below.
\[\begin{tikzcd}
	{\mathbf{Betti}(n)} & {\mathbf{Barc}(n)} & {\JS(\leq n+1, \langle 0 \rangle, n)} \\
	{\{\beta\}} & {\mathbf{Barc}(\beta)} & {\JS(\langle\delta(\beta)\rangle, \langle 0 \rangle, n)}
	\arrow["\varphi"', from=1-2, to=1-1]
	\arrow["\begin{array}{c} \sim \\ \sigma \end{array}"{marking, allow upside down}, from=1-2, to=1-3]
	\arrow[hook, from=2-1, to=1-1]
	\arrow[hook, from=2-2, to=1-2]
	\arrow["\varphi"', from=2-2, to=2-1]
	\arrow["\begin{array}{c} \sim \\ \sigma_{\beta} \end{array}"{marking, allow upside down}, from=2-2, to=2-3]
	\arrow[hook, from=2-3, to=1-3]
\end{tikzcd}\]

\begin{proof}
We first show that $\sigma_\beta$ is a bijection.
Let $\barc$ be a barcode with Betti curve $\beta$.
First, we show that $\sigma(\barc)$ is a valid magic juggling sequence.
For any $\barc$ and any $i$, \[\sum_j \delta(\barc^i)_j = 0, \]
so \cref{obs_valid_sequence} (i) is satisfied.
Moreover, \cref{obs_valid_sequence} (ii) is exactly \cref{lemma_bijection} (iii), so $\sigma(\barc)$ is a valid magic juggling sequence. 
Thus, by \cref{lemma_bijection} (i) and (ii), $\sigma(\barc)$ is a well-defined element of $\JS(\langle \delta(\beta) \rangle,\langle 0 \rangle, n)$.
Since $\delta(\beta)$ has length $n+1$, this shows that both $\sigma$ and $\sigma_\beta$ are well defined.

\cref{lemma_bijection} (iii) guarantees injectivity of $\sigma$, since we can read off $\barc$ from $\sigma(\barc)$.

To show surjectivity of $\sigma_\beta$, for each $T=(\mathbf{t}^0, \mathbf{t}^1, \dots, \mathbf{t}^n)\in \JS(\langle \delta(\beta) \rangle,\langle 0 \rangle, n)$, we find a barcode with Betti curve $\beta$ whose image under $\sigma$ is $T$. Write $\mathbf{t}^i = \langle t_1^i,t_2^i, \dots, t_h^i\rangle$ and let $\barc_T$ be the multiset of intervals where the multiplicity of the interval $[i,i+j)$ is $t^{i}_{j} - t^{i-1}_{j+1}$ (which is nonnegative by \cref{obs_valid_sequence} (ii)) for each $i\in \{1, \ldots, n\}$ and each $j\geq 0$.
The barcode $\barc_T$ has Betti curve $\beta$, since $\langle \delta(\barc_T) \rangle= \langle \delta(\barc_T^0) \rangle = \sigma(\barc_T)^0 = \langle \delta(\beta)\rangle$.
Write $\sigma(\barc_T)=(\mathbf{s}^0, \mathbf{s}^1, \dots, \mathbf{s}^n)$ and $\mathbf{s}^i = \langle s^i_1, s^i_2, \dots, s^i_{n+1}\rangle$.
We have $\sigma(\barc_T)=T$:
By \cref{lemma_bijection} (i), $\mathbf{s}^0 = \delta(\beta) = \mathbf{t}^0$, and by \cref{lemma_bijection} (iii), $s^{i}_{j} - s^{i-1}_{j+1} = b_{i,i+j} = t^{i}_{j} - t^{i-1}_{j+1}$ for all relevant $i$ and $j$. Since $\mathbf{s}^i$ is determined by $\mathbf{s}^{i-1}$ and the values $s^{i}_{j} - s^{i-1}_{j+1}$ for different $j$, and similarly for $\mathbf{t}^i$, this implies $\mathbf{s}^i = \mathbf{t}^i$ for all $i$.
Thus, $\sigma(\barc_T)=T$.

We have that $\sigma$ is injective and that $\sigma_\beta$ is bijective for each $\beta$.
Therefore, it only remains to show that for every juggling state of the form $\textbf{a} = \langle s_1,\dots,s_{n+1} \rangle$ such that $\JS(\textbf{a},\langle 0 \rangle, n)\neq \emptyset$, we have $\textbf{a} = \langle \delta(\beta) \rangle$ for some Betti curve $\beta=(\beta_1, \beta_2, \dots, \beta_{n})$.
We let $\beta_i = \sum_{j=1}^i s_j$ for $1\leq i\leq n$, which gives $\textbf{a} = \langle \delta(\beta) \rangle$ if $s_{n+1} = -\beta_n$.
This is equivalent to $\sum_{j=1}^{n+1} s_j = 0$.
By the assumption that $\JS(\textbf{a},\langle 0 \rangle, n)\neq \emptyset$, there is a valid juggling sequence $S$ starting at $\textbf{a}$ and ending at $\langle 0 \rangle$.
Let $S^i_{\geq j} = \sum_{k=j}^{h} s^i_k$, using the same notation as in \cref{obs_valid_sequence}.
By \cref{obs_valid_sequence}(i), $S^i_{\geq 1} = 0$ for all $i$, so in particular, $\sum_{j=1}^{n+1} s_j = 0$.

Finally, we need that all the $\beta_i$ are nonnegative for them to form a valid Betti curve.
Let $S\in \JS(\textbf{a},\langle 0 \rangle, n)$, and let $S^i_{\geq j} = \sum_{k=j}^{h} s^i_k$, using the same notation $s^i_k$ as in \cref{obs_valid_sequence}.
By \cref{obs_valid_sequence}(ii), $S^{i-1}_{\geq j+1}\leq S^i_{\geq j}$ for $1\leq i\leq n$ and $j\geq 1$.
We have $S^n_{\geq j} = 0$ for all $j$, and we already saw that $S^i_{\geq 1} = 0$ for all $i$.
Thus, we get
\[
S^i_{\geq j}\leq S^{i+1}_{\geq j-1}\leq \dots \leq 0
\]
for all $i$ and $j$, and therefore
\[
\beta_i = \sum_{j=1}^i s_j = S^0_{\geq 1} - S^0_{\geq i+1} \geq 0.\qedhere
\]
\end{proof}

Using \Cref{thm:magic_juggling_KPF} and \Cref{thm:bijection}, the relationship between the Kostant partition function and the barcode counting problem can now be obtained as a corollary.
\begin{corollary}\label{cor_1}
    If $\mu= \sum\limits_{i=1}^{n} \beta_i \alpha_{i}$ is a weight of the Lie algebra of type $A_n$, then 
    \[|\mathbf{Barc}(\beta)| = K(\mu).\]
\end{corollary}
\begin{proof}
    By \Cref{thm:bijection}, $\mathbf{Barc}(\beta)$ and $\JS(\langle\delta(\beta)\rangle, \langle 0 \rangle, n)$ have the same cardinality.
    The weight $\mu$ can be alternatively written as $\sum\limits_{i=1}^{n+1} \delta(\beta)_i \mathbf{e}_{i}$. If we let \[\mathcal{J}\coloneqq\JS(\langle \delta(\beta)_1, \ldots, \delta(\beta)_n \rangle, \langle \delta(\beta)_1 + \ldots+ \delta(\beta)_n \rangle, n),\] then by \Cref{thm:magic_juggling_KPF}, $K(\mu)= |\mathcal{J}|$.
    Notice that because 
    \[\delta(\beta)_{n+1} = - \sum_{i=1}^n \delta(\beta)_i,\] if we let $S \in \JS(\langle\delta(\beta)\rangle, \langle 0 \rangle, n)$, then the deletion of the $(n+1)$st entry from the first state of $S$ yields an element of $\mathcal{J}$. Since this entry never affects the balls that are thrown or caught in a sequence of length $n$, this deletion yields a bijection between $\JS(\langle\delta(\beta)\rangle, \langle 0 \rangle, n)$ and $\mathcal{J}$.
\end{proof}


\section{Discussion and Future Work}

As highlighted in \cite{hensel2021survey}, many applications tend to avoid the use of persistent homology directly and instead use the downstream Betti curves, which are simple to define, stable, and amenable to standard statistical techniques such as smoothing or averaging.
Our main result thus quantifies precisely how much information is lost by using the Betti curve.
As \Cref{fig:232_barcode_table} shows, the simple Betti curve
$\beta =(2,3,2)$ can be obtained by 13 different barcodes.
For another example, if $\beta=(2,3,1,1,1)$, computer calculation reveals that there are 32 different barcodes with this Betti curve.
By leveraging the new connections between TDA and algebraic combinatorics established in this paper, it's likely that many existing results can find new applications under this perspective.
Finally, as the Kostant partition function can be studied for very general Lie algebra representations, there is a natural question as to whether these connections can be leveraged to study inverse problems for multiparameter persistence.

\bibliography{ref}{}
\bibliographystyle{plain}

\end{document}